\documentclass[11pt]{article}

\usepackage{mathrsfs}
\usepackage{amsmath}
\usepackage{amsfonts}
\usepackage{color}
\usepackage{amssymb}
\usepackage{graphicx}
\usepackage{hyperref}
\usepackage{enumerate}
\usepackage[all]{xy}

\def\Aut{{\rm \textsf{Aut}}}
\def\Out{{\rm \textsf{Out}}}

\def\Soc{{\rm \textsf{Soc}}}
\def\Fix{{\rm \textsf{Fix}}}
\def\max{{\rm \textsf{max}}}

 \allowdisplaybreaks

\begin{document}

\newtheorem{problem}{Problem}

\newtheorem{theorem}{Theorem}[section]
\newtheorem{lemma}[theorem]{Lemma}
\newtheorem{corollary}[theorem]{Corollary}
\newtheorem{definition}[theorem]{Definition}
\newtheorem{conjecture}[theorem]{Conjecture}
\newtheorem{question}[theorem]{Question}
\newtheorem{proposition}[theorem]{Proposition}
\newtheorem{quest}[theorem]{Question}
\newtheorem{example}[theorem]{Example}

\newenvironment{proof}{\noindent {\bf
Proof.}}{\rule{2mm}{2mm}\par\medskip}

\newenvironment{proofof}{\noindent {\bf
Proof of Theorem \ref{th1}}}{\rule{2mm}{2mm}\par\medskip}

\newcommand{\remark}{\medskip\par\noindent {\bf Remark.~~}}
\newcommand{\pp}{{\it p.}}
\newcommand{\de}{\em}

\title{  {Block-transitive point-primitive automorphism groups of Steiner $3$-designs} }

\author{Yunsong Gan, Weijun Liu\thanks{Corresponding author.
\newline E-mail addresses: songsirr@126.com (Y. Gan), wjliu6210@126.com (W. Liu). } \\
{\small School of Mathematics and Statistics, Central South University} \\
{\small New Campus, Changsha, Hunan, 410083, P.R. China. } }

\maketitle

\vspace{-0.5cm}

\begin{abstract}
This paper studies the long-standing open problem of the reduction of Steiner $3$-designs admitting a block-transitive automorphism group. We prove that if $G$ acts as a block-transitive point-primitive automorphism group of a nontrivial Steiner $3$-design, then $G$ is affine or almost simple.
\end{abstract}


\section{Introduction}

\label{sec1}

A $t$-$(v,k,\lambda)$ {\it design} $\mathcal{D}=(\mathcal{P},\mathcal{B})$ with parameters $(v,b,r,k,\lambda)$ is an incidence structure consisting of a set $\mathcal{P}$ of $v$ {\it points} and a set $\mathcal{B}$ of $b$ {\it blocks} such that every block is incident with $k$ points, every point is incident with $r$ blocks, and any $t$ distinct points are incident with exactly $\lambda$ blocks.
If furthermore $t<k<v-t$ holds, then $\mathcal{D}$ is said to be {\it nontrivial}.
All designs in this paper are assumed to be nontrivial.
A {\it flag} of $\mathcal{D}$ is an incident point-block pair $(\alpha,B)$, where $\alpha$ is a point and $B$ is a block incident with $\alpha$.

An {\it automorphism} of a $t$-design $\mathcal{D}$ is a permutation of the points which preserves the blocks.
The set of all automorphisms of $\mathcal{D}$ forms the {\it full automorphism group}, denoted by $\Aut(\mathcal{D})$.
If $G$ is a subgroup of $\Aut(\mathcal{D})$, then $G$ is called the {\it automorphism group} of $\mathcal{D}$.
Moreover, $G$ is said to be {\it point-primitive} (respectively {\it block-transitive, flag-transitive}) if it acts primitively on the points (respectively transitively on the blocks, transitively on the flags).
It is well known that if $G$ is block-transitive, then $G$ is also point-transitive (Block Lemma \cite{Bl67}).
For short, we say the design $D$ is, for example, block-transitive if $G$ is block-transitive.

A $t$-$(v,k,\lambda)$ design with $\lambda=1$ is called a {\it Steiner t-design}.
As a consequence of the classification of the finite simple groups, it has been possible to characterize Steiner $t$-designs with sufficiently strong transitivity properties.
Let $\mathcal{D}$ be a Steiner $t$-design admitting a point-primitive automorphism group $G$.
The O'Nan-Scott theorem shows that $G$ is permutation isomorphic to a group that is either of affine type, almost simple type, product type, simple diagonal type or twisted wreath type(see \cite{Lie88} for more details).
In recent decades, there has been a lot of research on the automorphism groups of Steiner $2$-designs.
One of the most significant contributions was the reduction of block-transitive point-primitive Steiner $2$-designs \cite{Camina96}.
Since then the effort has been to classify the block-transitive examples admitting the specific groups. The papers \cite{Lihuiling01} and \cite{Liuwj03} contain some discussions of the automorphism groups of block-transitive Steiner $2$-designs.

For a $t$-$(v,k,\lambda)$ design, Cameron and Praeger prove the following result in 1993.
\begin{theorem}{\rm\cite{Cameron93}}  \label{th1.1}
Let $\mathcal{D}=(\mathcal{P},\mathcal{B})$ be a $t$-$(v,k,\lambda)$ design with $t\geq2$ and let $G$ be a subgroup of $\Aut(\mathcal{D})$.
Then the following hold.
\begin{enumerate}[{\rm(i)}]

\item  If $G$ is block-transitive on $\mathcal{D}$, then $G$ is $\lfloor t/2 \rfloor$-homogeneous on points.

\item  If $G$ is flag-transitive on $\mathcal{D}$, then $G$ is $\lfloor (t+1)/2 \rfloor$-homogeneous on points.

\end{enumerate}
\end{theorem}
According to this result, if the automorphism group $G$ acts flag-transitively on a Steiner $3$-design, then $G$ is either point-primitive of affine or almost simple type.
In 2005, Huber \cite{Huber05} completely classified flag-transitive Steiner $3$-designs.
However, so far there are few conclusions under the weaker assumption that the automorphism group $G$ acts block-transitively on a Steiner $3$-design.
In this paper we consider that $\mathcal{D}$ is a nontrivial Steiner $3$-design admitting a block-transitive point-primitive automorphism group $G$.
Depending on the classification of finite simple groups and the O'Nan-Scott theorem, we prove the following theorem.

\begin{theorem}\label{th1}
Let $G$ act as a block-transitive point-primitive automorphism group of a nontrivial Steiner $3$-design. Then $G$ is affine or almost simple.
\end{theorem}

\section{Notation and preliminaries}
\label{sec2}
In this section, we give some definitions and results which will be used in our analyses.

We follow \cite{Be99} for fundamental notations on designs.
Let $\mathcal{D}=(\mathcal{P},\mathcal{B})$ be a $t$-$(v,k,\lambda)$ design with $t\geq2$.
The parameter $\lambda_2$ denotes the total number of blocks incident with a given pair of distinct points.
For every point $\alpha\in\mathcal{P}$ , the {\it star} of $\alpha$ is the set $st(\alpha)=\{B\in \mathcal{B}\mid\alpha\in B\}$.
Let $G$ be a subgroup of $\Aut(\mathcal{D})$.
The {\it socle} $\Soc(G)$ of $G$ is the product of all minimal normal subgroups of $G$.
For $H\leq G$, let $\Fix_{\mathcal{P}}(H)$ denote the set of fixed points of $H$ in $\mathcal{P}$.
We call the automorphism group $G$ {\it semi-regular} if the identity is the only element that fixes a point of $\mathcal{P}$.
If additionally $G$ is transitive, then we call $G$ {\it regular}.
For $\{\alpha_1,\ldots,\alpha_m\}\subseteq\mathcal{P}$ and $B\in\mathcal{B}$, let $G_{\alpha_1\ldots\alpha_m}$ be the pointwise stabilizer of $\{\alpha_1,\ldots,\alpha_m\}$ in $G$ and let $G_B$ be the setwise stabilizer of $B$ in $G$.
For a finite transitive permutation group, the length of an orbit of the point stabilizer is said to be {\it subdegree}.

Let $m$ and $n$ be integers and $p$ a prime.
Then $(m,n)$ is the greatest common divisor of $m$ and $n$.
We write $m\mid n$ if $m$ divides $n$ and $p^m\parallel n$ if $p^m$ divides $n$ but $p^{m+1}$ does not divide $n$.
For any $x\in\mathbb{R}$, let $\lfloor x\rfloor$ denote the greatest positive integer which is at most $x$.
All other notations are standard.

We first collect some useful results on Steiner $3$-designs and permutation groups.

\begin{lemma}{\rm\cite{Be99}}  \label{lem2.2}
If $\mathcal{D}=(\mathcal{P},\mathcal{B})$ is a $t$-$(v,k,\lambda)$ design, then the number of blocks containing any given point of $\mathcal{D}$ is a constant $r$ and the following hold:
\begin{enumerate}[{\rm(i)}]

\item  $vr=bk$;

\item  $\binom{v}{t}\lambda=b\binom{k}{t}$.

\end{enumerate}
\end{lemma}

\begin{lemma}{\rm\cite{Be99}} \label{lem2.3}
Let $\mathcal{D}=(\mathcal{P},\mathcal{B})$ be a $t$-$(v,k,\lambda)$ design.
Then, for any integer $s$ satisfying $1\leq s\leq t$, $\mathcal{D}=(\mathcal{P},\mathcal{B})$ is also an $s$-$(v,k,\lambda_s)$ design, where
\[\lambda_s=\lambda\frac{(v-s)(v-s-1)\cdots(v-t+1)}{(k-s)(k-s-1)\cdots(k-t+1)}.\]
In particular, if $\mathcal{D}$ is a Steiner $3$-design, then $(k-2)\lambda_2=v-2$.
\end{lemma}

The following lemma gives an upper bound for the block size $k$ of a Steiner $3$-design.
\begin{lemma}{\rm{(\cite[Corollary 5]{Huber05})}}\label{lem2.4}
Let $\mathcal{D}=(\mathcal{P},\mathcal{B})$ be a nontrivial Steiner $3$-design.
Then the block size $k$ can be estimated by
\[k\leq\left\lfloor\sqrt{v}+\frac{3}{2}\right\rfloor. \]
\end{lemma}

Considering the fact that $(q+1)q(q-1)$ is divisible by both $3$ and $4$ for any prime power $q>2$, we obtain the following lemma.

\begin{lemma}  \label{lem2.5}
Let $T\ncong Sz(q)$ be a nonabelian simple group, where $Sz(q)$ is the Suzuki group.
Then $|T|$ is divisible by $4$ and $T$ contains an element with order $3$.
\end{lemma}

\begin{lemma}  \label{lem2.6}
Let $\mathcal{D}=(\mathcal{P},\mathcal{B})$ be a Steiner $3$-design admitting an automorphism group $G$.
If there exists an element $s$ of $G$ with order $3$ and $s$ fixes no points, then $k$ divides $v$.
\end{lemma}

\begin{proof}
By counting the number of $3$-cycles of $s$, the number of the blocks fixed by $s$ equals $v/k$.
This implies that the set of points is the union of disjoint blocks and thus $k$ divides $ v$.
\end{proof}

\begin{lemma}{\rm\cite{Camina84}}  \label{lem2.8}
Let $X$ be a permutation group on a set $\Delta$.
Assume that $p^e$ is the highest power of some prime $p$ which divides $|\Delta|$ and assume that $p^e$ divides $[X:X_{\alpha\beta}]$ for all $\alpha,\beta\in\Delta$, $\alpha\neq\beta$.
Then each orbit of $X$ on $\Delta$ has size divisible by $p^e$.
\end{lemma}

The next somewhat technical lemma is crucial to the proof of Theorem \ref{th1}.
\begin{lemma}  \label{lem2.7}
Let $\mathcal{D}=(\mathcal{P},\mathcal{B})$ be a nontrivial Steiner $3$-design admitting a block-transitive automorphism group $G$.
Let $B$ be a block.
If $k\mid v$ and $4\mid v$, then $G$ is flag-transitive.
\end{lemma}

\begin{proof}
For any $\alpha$, $\beta\in B$, as $G$ acts block-transitively and point-transitively on $\mathcal{D}$, we get
\[[G:G_{\alpha\beta B}]=[G:G_{\alpha}][G_{\alpha}:G_{\alpha\beta B}]=v[G_{\alpha}:G_{\alpha\beta B}]\]
and
\[[G:G_{\alpha\beta B}]=[G:G_B][G_B:G_{\alpha\beta B}]=b[G_B:G_{\alpha\beta B}].\]
Combining these two equations with Lemma \ref{lem2.2}, we obtain
\[k(k-1)\Bigm| \frac{(v-1)(v-2)[G_{B}:G_{\alpha\beta B}]}{k-2}.\]
When $k$ is odd, $k\mid v$ implies $\left(k,(v-1)(v-2)\right)=1$ and thus $k\mid[G_{B}:G_{\alpha\beta B}]$.
When $k$ is even, notice that $2\mid(k-2)$ and $2\parallel (v-2)$.
It follows that $\left(k,\frac{(v-1)(v-2)}{(2,k-2)}\right)=1$.
Hence, $k\mid[G_{B}:G_{\alpha\beta B}]$.

However, $G_B$ acts on the $k$ points of the block $B$.
Let $p$ be a prime dividing $k$ and let $p^e\|k$.
Then we have $p^e\mid[G_{B}:G_{\alpha\beta B}]$, for any $\alpha$, $\beta\in B$.
Therefore, by Lemma \ref{lem2.8}, $p^e$ divides the size of the orbits of $G_B$ on $B$.
But this is true for each prime divisor of $k$ and so $G_B$ is transitive on the points of $B$.
Thus $G$ is flag-transitive as claimed.
\end{proof}

\begin{lemma}  \label{lem2.9}
Let $\mathcal{D}=(\mathcal{P},\mathcal{B})$ be a Steiner $3$-design admitting a block-transitive automorphism group $G$.
Then the following hold:
\begin{enumerate}[{\rm(i)}]

\item  $(v-1)(v-2)$ divides $k(k-1)(k-2)|G_\alpha|$ for any $\alpha\in\mathcal{P}$;

\item  $(v-1)(v-2)$ divides $k(k-1)(k-2)d(d-1)$ for all nontrivial subdegrees $d$ of $G$.

\end{enumerate}
\end{lemma}

\begin{proof}
Let $B$ be a block containing the point $\alpha$.
As $G$ acts block-transitively and point-transitively on $\mathcal{D}$, we have
\[[G:G_{\alpha B}]=[G:G_{\alpha}][G_{\alpha}:G_{\alpha B}]=v[G_{\alpha}:G_{\alpha B}]\]
and
\[[G:G_{\alpha B}]=[G:G_B][G_B:G_{\alpha B}]=b[G_B:G_{\alpha B}].\]
Together with Lemma \ref{lem2.2}, the above two equations lead to
\[[G_{\alpha}:G_{\alpha B}]=\frac{b[G_{B}:G_{\alpha B}]}{v}=\frac{(v-1)(v-2)[G_{B}:G_{\alpha B}]}{k(k-1)(k-2)},\]
which implies that $(v-1)(v-2)$ divides $k(k-1)(k-2)|G_\alpha|$.

Assume that $\mathcal{O}_1,\mathcal{O}_2,\ldots,\mathcal{O}_w$ are the $w$ orbits of $G_\alpha$ on $st(\alpha)$ and $B_i$ is the representative of $\mathcal{O}_i$ for $i=1,2,\ldots,w$.
Thus by Lemma \ref{lem2.2}, we get
\[|\mathcal{O}_i|=[G_{\alpha}:G_{\alpha B_i}]=\frac{b[G_{B_i}:G_{\alpha B_i}]}{v}=\frac{r[G_{B_i}:G_{\alpha B_i}]}{k}.\]
Suppose that $\Gamma\neq\{\alpha\}$ is a nontrivial $G_{\alpha}$-orbit with $|\Gamma|=d$.
Let $\mu_i=|\Gamma\cap B_i|$. Obviously, $\mu_i$ is independent of the choice of the representative.
Counting the size of set $\{(\{\beta,\gamma\},B)\mid B\in st(\alpha),\{\beta,\gamma\}\subseteq B\cap\Gamma\}$ in two ways, we obtain
\[|\mathcal{O}_1|\binom{\mu_1}{2}+\cdots+|\mathcal{O}_w|\binom{\mu_w}{2}=\binom{d}{2}.\]
Therefore,
\[\frac{r[G_{B_1}:G_{\alpha B_1}]}{k}\binom{\mu_1}{2}+\cdots+\frac{r[G_{B_w}:G_{\alpha B_w}]}{k}\binom{\mu_w}{2}=\binom{d}{2}.\]
Hence, we get  that $r$ divides $kd(d-1)$.
By Lemma \ref{lem2.2} again, we conclude that $(v-1)(v-2)$ divides $k(k-1)(k-2)d(d-1)$.
\end{proof}

\section{Proof of Theorem \ref{th1}}
In this section, we shall always assume that $\mathcal{D}=(\mathcal{P},\mathcal{B})$ is a nontrivial Steiner $3$-design with parameters $(v,b,r,k,\lambda)$ and $G\leq\Aut(\mathcal{D})$ is block-transitive point-primitive.
The three types of the primitive group $G$ on $\mathcal{P}$ which we need to examine are the simple diagonal type, twisted wreath type and the product type.
Firstly, we prove the following proposition.

\begin{proposition}  \label{pro3.1}
Let $G$ act as a block-transitive point-primitive automorphism group of a nontrivial Steiner $3$-design $\mathcal{D}=(\mathcal{P},\mathcal{B})$.
Then $G$ is not of simple diagonal type or twisted wreath type.
\end{proposition}

\begin{proof}
Let $B$ be a block containing the point $\alpha$.
Assume that $G$ is of simple diagonal type or twisted wreath type with $\Soc(G)=T_1\times T_2\times \cdots\times T_m\cong T^m$, where $T_i\cong T$ is a nonabelian simple group.
Then, $T_1$ is semi-regular and $v=|T|^{m-1}$ or $|T|^m$, $m\geq2$.
We distinguish two cases according as $T\cong Sz(q)$ or not, where $Sz(q)$ is the Suzuki group.

Let $T\ncong Sz(q)$. By Lemma \ref{lem2.5}, $v$ is divisible by $4$ and there exists an element $s\in T_1$ with order $3$.
As $T_1$ is semi-regular, there is no point of $\mathcal{P}$ fixed by $s$.
Hence by Lemma \ref{lem2.6}, $k$ divides $v$.
Combining this with Lemma \ref{lem2.7}, we obtain that $G$ is flag-transitive, yielding a contradiction to Theorem \ref{th1.1}.

Let $T\cong Sz(q)$.
Then $v\geq|T|\geq29210.$
First we suppose that $G$ is a primitive group of twisted wreath type on $\mathcal{P}$.
Then $v=|T|^{m}$ and $G\cong {_Q}{}H\rtimes P$, where $P=G_{\alpha}$ is a transitive permutation group on $\{1,2,\cdots,m\}$ with $m\geq6$.
Moreover, $G_{\alpha}$ has an orbit $\Gamma$ in $\mathcal{P}\setminus\{\alpha\}$ with $|\Gamma|\leq m|T|$.
By Lemma \ref{lem2.9} (ii), we have
\[(|T|^m-1)(|T|^m-2)\leq k(k-1)(k-2)|\Gamma|(|\Gamma|-1)\leq k(k-1)(k-2)m|T|(m|T|-1),\]
which contradicts Lemma \ref{lem2.4}.

Next, let $G$ be a primitive group of simple diagonal type on $\mathcal{P}$.
Then, $v=|T|^{m-1}$.
Furthermore, $G_{\alpha}$ is isomorphic to a subgroup of $\Aut(T)\times S_m$ and has an orbit $\Gamma'$ in $\mathcal{P}\setminus\{\alpha\}$ with $|\Gamma'|\leq m|T|$, where $S_m$ is the symmetric group of degree $m$.
Similarly, Lemma \ref{lem2.9} (ii) and \ref{lem2.4} lead to
\begin{align*}
\left(|T|^{m-1}-1\right)\left(|T|^{m-1}-2\right)&\leq k\left(k-1\right)\left(k-2\right)|\Gamma'|\left(|\Gamma'|-1\right)\\
&\leq\left(|T|^{\frac{m-1}{2}}+2\right)\left(|T|^{\frac{m-1}{2}}+1\right)|T|^{\frac{m-1}{2}}m|T|\left(m|T|-1\right),
\end{align*}
and thus $m\leq5$.
By Lemma \ref{lem2.9} (i), we obtain
\[\left(|T|^{m-1}-1\right)\left(|T|^{m-1}-2\right)\Bigm| k(k-1)(k-2)|S_m||\Aut(T)|\Bigm| k(k-1)(k-2)|S_m||T||\Out(T)|,\]
where $\Out(T)$ is the outer automorphism group of $T$.
Note that $|\Out(Sz(q))|\leq |Sz(q)|^{\frac{1}{5}}$, $\left(|T|^{m-1}-1,|T|\right)=1$ and $\left(|T|^{m-1}-2,|T|\right)=2$, which imply
\[\left(|T|^{m-1}-1\right)\left(|T|^{m-1}-2\right)\leq2k(k-1)(k-2)|S_m||\Out(T)|,\]
where $m\leq5$. This contradicts Lemma \ref{lem2.4}.
\end{proof}

To complete the proof of Theorem \ref{th1}, we need to show that the automorphism group with product type does not occur.
For proving this, we give some notations and two simple lemmas.

Let $\mathcal{D}=(\mathcal{P},\mathcal{B})$ be a nontrivial Steiner $3$-design admitting a point-primitive automorphism group $G$ with product type.
Further, we assume $G\leq(\Psi_1\times\cdots\times\Psi_m)\rtimes S_m=\Psi\wr S_m$ with $m\geq2$, where $S_m$ is the symmetric group of degree $m$ and $\Psi_i\cong\Psi$ acts primitively on a set $\mathcal{Q}$ with degree $w\geq5$.
Hence, $v=w^m$ and
\[\mathcal{P}=\{(q_1,q_2,\ldots,q_m)\mid q_i\in\mathcal{Q},i=1,2,\ldots,m\}.\]
Moreover, $\Psi_i\cong\{(\theta_{1},\theta_{2},\ldots,\theta_{m})\mid\theta_{i}\in\Psi,\theta_{j}=1$ for any $j\neq i\}$ and the action of the element $(\theta_{1},\theta_{2},\ldots,\theta_{m})\in G$ on points has the form
\[(q_1,q_2,\ldots,q_m)^{(\theta_{1},\theta_{2},\ldots,\theta_{m})}=(q_1^{\theta_{1}},q_2^{\theta_{2}},\ldots,q_d^{\theta_{m}}),\]
where $q_i\in\mathcal{Q}$, $\theta_{i}\in\Psi$, $1\leq i\leq m$.

\begin{lemma}  \label{lem3.2}
Let $G$ act block-transitively on a Steiner $3$-design $\mathcal{D}=(\mathcal{P},\mathcal{B})$.
If $H\neq1$ is a subgroup of $G$ and has an orbit $\Gamma$ in the point set $\mathcal{P}$ with $|\Gamma|\geq3$, then
\[|\Fix_{\mathcal{P}}(H)|\leq\frac{2(v-k)}{k-2}+k-2.\]
\end{lemma}

\begin{proof}
Let $\{\alpha,\beta,\gamma\}\subseteq\Gamma$.
By Lemma \ref{lem2.3}, the number of the blocks containing $\{\alpha,\beta\}$ is $\lambda_2$.
Notice that $\mathcal{D}$ is a block-transitive Steiner $3$-design, which implies that $\{\alpha,\beta\}$ is contained in at most one block fixed by $H$.
If there is no such block, then $H$ fixes at most two points in each block containing $\{\alpha,\beta\}$.
So $|\Fix_{\mathcal{P}}(H)|\leq2\lambda_2$.
Otherwise, $H$ can fix at most $k-3$ points in the fixed block and so $|\Fix_{\mathcal{P}}(H)|\leq2(\lambda_2-1)+(k-3)$.
Thus, applying Lemma \ref{lem2.3}, we have
\[|\Fix_{\mathcal{P}}(H)|\leq \max\{2\lambda_2,\, 2(\lambda_2-1)+(k-3)\}\leq\frac{2(v-k)}{k-2}+k-2.\]
This completes the proof.
\end{proof}

\begin{lemma}  \label{lem3.3}
Let $\mathcal{D}=(\mathcal{P},\mathcal{B})$ be a nontrivial Steiner $3$-design admitting a block-transitive point-primitive automorphism group $G$ with product type as described before Lemma \ref{lem3.2}.
Then the following hold.
\begin{enumerate}[{\rm(i)}]

\item  If $m\geq4$, then $k\leq2w+2$.

\item  If $m=3$ and $w\geq11$, then $k\leq3w+2$.

\end{enumerate}
\end{lemma}

\begin{proof}
Let $K=(\Psi_1)_p\leq G$, where $p=(\overline{q_1},\overline{q_2},\ldots,\overline{q_m})\in\mathcal{P}$.
Clearly, $v=w^m$ and $|\Fix_{\mathcal{P}}(K)|\geq w^{m-1}.$
We consider the action of $K$ on $\{(q_1,\overline{q_2},\ldots,\overline{q_m})\mid q_1\in\mathcal{Q}\}$, which is permutation isomorphic to the action of $\Psi_{\overline{q_1}}$ on $\mathcal{Q}\setminus\{\overline{q_1}\}$.
Because $\Psi$ is a nonregular primitive group acting on $\mathcal{Q}$ with a nonabelian socle, there exists a subdegree $d\geq3$ of $\Psi$.
By Lemma \ref{lem3.2}, we get
\[w^{m-1}\leq|\Fix_{\mathcal{P}}(K)|\leq\frac{2(w^m-k)}{k-2}+k-2.\]
We can rewrite this as an inequality in $k$,
\[0\leq k^2-(w^{m-1}+6)k+2w^m+2w^{m-1}+4.\]
The discriminant of this, when considered as a quadratic equation, is given by $D=w^{2m-2}-8w^m+4w^{m-1}+20$.
If $m\geq4$, then
\[D=(w^{m-1}-4w)^2+4w^{m-1}-16w^2+20>(w^{m-1}-4w)^2.\]
Hence we get that either $k>w^{m-1}-2w+3$ or $k<2w+3$.
By Lemma \ref{lem2.4}, only $k<2w+3$ can occur.
If $m=3$ and $w\geq11$, the discriminant satisfies
\[D=(w^2-6w)^2+4w^3-32w^2+20>(w^2-6w)^2.\]
Using Lemma \ref{lem2.4} again, we obtain $k<3w+3$.
\end{proof}

\begin{proposition}  \label{pro3.4}
Let $G$ act as a block-transitive point-primitive automorphism group of a nontrivial Steiner $3$-design $\mathcal{D}=(\mathcal{P},\mathcal{B})$.
Then $G$ is not of product type.
\end{proposition}

\begin{proof}
Let the automorphism group $G$ act on $\mathcal{D}$ with product type as described before Lemma \ref{lem3.2}.
Then, $v=w^m$.
Consider the block $B_1$ which contains three points $(u_1,\ldots,u_{m-1},x)$, $(u_1,\ldots,u_{m-1},y)$ and $(u_1,\ldots,u_{m-1},z)$.
Let $B_2$ be another block.
Then there exists $\gamma \in G$ such that $B_2=B_1^{\gamma}$ and $\gamma$ has the form $(\theta_{1},\theta_{2},\ldots,\theta_{m})\sigma$, where each $\theta_{i}\in\Psi$ and $\sigma\in S_m$.
It can be seen that each block contains at least two points which disagree in only one position, and clearly, the number of such pairs in a block is independent of the choice of the block.
Let this number be $x_2$.
Analogously, we define $x_3$ as the number of the three points in a block which disagree in exactly one position.
Then $x_2\geq3$ and $x_3\geq1$.
Counting the total number of such two or three points in $\mathcal{P}$ in two ways, together with Lemma \ref{lem2.2} we have
\[\frac{x_2b}{\lambda_2}=x_2\frac{w^m(w^m-1)}{k(k-1)}=\frac{w^mm(w-1)}{2}\]
and
\[x_3b=x_3\frac{w^m(w^m-1)(w^m-2)}{k(k-1)(k-2)}=\frac{w^mm(w-1)(w-2)}{3!}.\]
By simplifying these two equations, we get
\begin{equation} \label{eq1}
mk(k-1)=2x_2(w^{m-1}+w^{m-2}+\cdots+1)
\end{equation}
and
\begin{equation} \label{eq2}
mk(k-1)(k-2)=6x_3(w^{m-1}+w^{m-2}+\cdots+1)\frac{w^{m-2}}{w-2}.
\end{equation}

First we deal with the case when $m\geq3$.
If $m\geq4$, then by Lemma \ref{lem3.3}, we obtain
\[mk(k-1)\leq m(2w+2)(2w+1)<2(w^{m-1}+w^{m-2}+\cdots+1),\]
which contradicts Equation (\ref{eq1}).
If $m=3$ and $w\geq11$, then by Lemma \ref{lem3.3} again, we have
\[3k(k-1)(k-2)\leq9w(3w+1)(3w+2)<6(w^{m-1}+w^{m-2}+\cdots+1)\frac{w^{m-2}}{w-2},\]
yielding a contradiction to Equation (\ref{eq2}).
When $m=3$ and $5\leq w\leq11$, it can be easily calculated that there is no solution of $k$ satisfying Equations (\ref{eq1}), (\ref{eq2}) and Lemma \ref{lem2.4}.

Finally, it remains to consider the case when $m=2$. Now $v=w^2$ and Equation (\ref{eq1}) becomes
\[k(k-1)=x_2(w+1).\]
Furthermore, by Lemma \ref{lem2.4}, we get $k\leq w+1$ and thus $x_2<k$.

Let $B=\{(p_1,q_1),(p_2,q_2),\cdots,(p_k,q_k)\}$ be a block with $p_1=p_2=p_3$.
Note that neither all $p_i$ nor all the $q_i$ can be the same for otherwise $x_2=\binom{k}{2}\geq k$.
So by rearranging the members of $B$, we can suppose that there exists an $i\geq4$ such that $p_1=p_2=\cdots=p_{i-1}\neq p_j$ for any $j\geq i$.
As $(p_1,q_1)$, $(p_2,q_2)$ and $(p_3,q_3)$ are fixed by $(\Psi_1)_{p_1}$, we have that $(\Psi_1)_{p_1}$ fixes $B$.
Moreover, $\Psi$ is a nonregular primitive group acting on $\mathcal{Q}$ with a nonabelian socle, which implies that there exists a subdegree $d\geq3$ of $\Psi$ on $\mathcal{Q}$.
Hence it can be assumed that there are three elements $(p_i,q_i),(r,q_i),(s,q_i)\in B$, where $r,s\in\mathcal{Q}$ are in the orbit of $p_i$ under the action of $\Psi_{p_1}$.
It follows that $(\Psi_2)_{q_i}$ also fixes $B$ and $k\geq6$.

Consider the action of $(\Psi_1)_{p_1}$ and $(\Psi_2)_{q_i}$ on $B$.
There are two possibilities for $B$, one of which is that
\[B=\{(p_1,q_1),(p_1,q_2),\ldots,(p_1,q_{i-1}),(p_i,q_i),(p_{i+1},q_i),\ldots,(p_k,q_i)\}.\]
Accordingly, depending on whether $q_i=q_j$ for some $j<i$,
\[x_2=\binom{i-1}{2}+\binom{k-i+2}{2}\quad {\rm or} \quad x_2=\binom{i-1}{2}+\binom{k-i+1}{2}.\]
However, both are contradictory to $x_2<k$.

The remaining possibility is that there exists a point $(p,q)\in B$ with $p\neq p_1$ and $q\neq q_i$.
In fact, for any $(p,q)\in B$ with $p\neq p_1$ and $q\neq q_i$, by the action of $(\Psi_1)_{p_1}$, there are at least three points of $B$ with $q$ in the second coordinate, say $(p,q)$, $(p',q)$ and $(p'',q)$.
Thus, $B$ can be written as
\[B=\{(p_1,q_1),(p_1,q_2),\ldots,(p_1,q_{i-1}),(p_i,q_i),(r,q_i),(s,q_i)\ldots,(p,q),(p',q),(p'',q),\ldots\}.\]
Assume that $\mathcal{Q}=\{o_1,o_2,\ldots,o_w\}$ and $a_j$ is the number of occurrences of $o_j$ at the second coordinate of members of $\{(x,y)\in B\mid x\neq p_1\}$.
Then $\sum\limits_{j=1}^w a_j=k-i+1$.
From the discussion above, we know that $a_j=0$ or $a_j\geq3$ for $j=1,2,\ldots,w$.
Therefore, we obtain
\[x_2\geq\binom{i-1}{2}+\sum_{j=1}^w \binom{a_j}{2}\geq k,\]
which is a contradiction.
\end{proof}

\begin{proofof}
It follows from Proposition \ref{pro3.1} and \ref{pro3.4}.
\end{proofof}

\section*{Acknowledgments}
The authors sincerely thank Zheng Huang, Wei Jin and Yongtao Li for consistent encouragement and valuable advice.

\end{document}